\documentclass[11pt]{article}
\usepackage{amsfonts}
\usepackage{latexsym}
\usepackage{cite}
\usepackage{amsmath,amsfonts,latexsym,amssymb}
\usepackage[mathscr]{eucal}
\usepackage{cases}
\usepackage{amsthm}

\usepackage[bf,small]{caption2}
\usepackage{float}
\usepackage{graphicx}
\usepackage{amsmath}
\usepackage{amssymb}
\usepackage[all]{xy}

\newtheorem{theorem}{theorem}[section]
\newtheorem{thm}[theorem]{Theorem}
\newtheorem{lem}[theorem]{Lemma}

\newtheorem{rmk}[theorem]{Remark}
\newtheorem{nota}[theorem]{Notation}

\begin{document}

\title{\textbf{Trace-free ${\rm SL}(2,\mathbb{C})$-representations of Montesinos links}}
\author{\Large Haimiao Chen
\footnote{Email: \emph{chenhm@math.pku.edu.cn}. The author is supported by NSFC-11401014.} \\
\normalsize \em{Beijing Technology and Business University, Beijing, China}}
\date{}
\maketitle

\begin{abstract}
  Given a link $L$, a representation $\pi_1(S^{3}-L)\to {\rm SL}(2,\mathbb{C})$ is {\it trace-free} if the image of each meridian has trace zero.
  We determine the conjugacy classes of trace-free representations when $L$ is a Montesinos link.

\end{abstract}

\section{Introduction}

Given a link $L\subset S^{3}$ and a linear group $G$, a {\it trace-free} (or {\it traceless}) $G$-representation of $L$ means a homomorphism $\pi_1(S^3-L)\to G$ sending each meridian to an element of trace zero. Dated back to 1980,
Magnus \cite{Mag80} used trace-free ${\rm SL}(2,\mathbb{C})$-representations to prove the faithfulness of a representation of braid groups in the automorphism groups of the rings generated by the characters functions on free groups.
Lin \cite{Lin92} used trace-free ${\rm SU}(2)$-representations to define a Casson-type invariant of a knot $K$, and showed it to equal half of the signature of $K$.
More interestingly, Kronheimer and Mrowka \cite{KM11} observed that for some knots $K$, its Khovanov homology is isomorphic to the ordinary homology of the space $R(K)$ of conjugacy classes of trace-free representations of $K$. In this context, Zentner \cite{Zen11} determined $R(K)$ when $K$ belongs to a class of classical pretzel knots. For related works, one may refer to \cite{FKC13}, \cite{HHK13}, etc.

There are relatively few results on trace-free ${\rm SL}(2,\mathbb{C})$-representations.  For a knot $k$, Nagasato \cite{Nag13} gave a set of polynomials whose zero locus is exactly the trace-free characters of irreducible ${\rm SL}(2,\mathbb{C})$-representations of $K$. Nagasato and Yamaguchi \cite{NY12} investigated trace-free ${\rm SL}(2,\mathbb{C})$-representations of $\pi_1(\Sigma-K)$ (where $K$ is a knot in an integral 3-sphere $\Sigma$), and related to those of $\pi_1(B_2)$, where $B_2$ is the 2-fold cover of $\Sigma$ branched along $K$.

In this paper,  for each Montesinos link, we completely determine the trace-free ${\rm SL}(2,\mathbb{C})$-representations by given explicit formulas.

\bigskip

Let ${\rm SL}_{0}(2,\mathbb{C})=\{X\in{\rm SL}(2,\mathbb{C})\colon {\rm tr}(X)=0\}$.
Note that each $X\in{\rm SL}_{0}(2,\mathbb{C})$ satisfies $X^{-1}=-X$.

By a ``tangle" we simultaneously mean an unoriented tangle diagram and the tangle it represents.
Given a tangle $T$, let ${\rm Dar}(T)$ denote the set of directed arcs of $T$, (each arc gives two directed arcs).
By a {\it (trace-free) representation} of a tangle $T$, we mean a map $\rho:{\rm Dar}(T)\to {\rm SL}_{0}(2,\mathbb{C})$ such that $\rho(x^{-1})=\rho(x)^{-1}$ for each $x\in {\rm Dar}(T)$ and at each crossing illustrated in Figure \ref{fig:crossing}, $\rho(z)=\rho(x)\rho(y)\rho(x)^{-1}$.
To present such a representation, it is sufficient to give each arc a direction and label an element of ${\rm SL}_{0}(2,\mathbb{C})$ beside it.

\begin{figure} [h]
  \centering
  \includegraphics[width=0.3\textwidth]{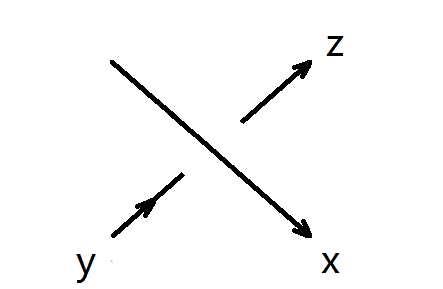}\\
  \caption{A representation satisfies $\rho(z)=\rho(x)\rho(y)\rho(x)^{-1}$ at each crossing} \label{fig:crossing}
\end{figure}

\begin{figure} [h]
  \centering
  \includegraphics[width=0.33\textwidth]{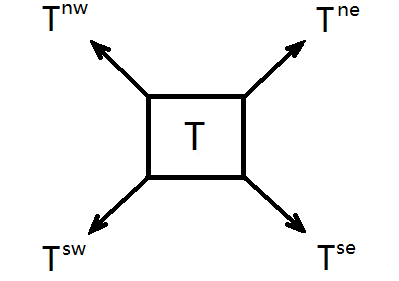}\\
  \caption{A tangle $T\in\mathcal{T}_2^2$, with the four ends directed outwards} \label{fig:general}
\end{figure}

Let $\mathcal{T}_2^2$ denote the set of tangles $T$ with four ends which, when directed outwards, are denoted by $T^{{\rm nw}}, T^{{\rm sw}}, T^{{\rm ne}}, T^{{\rm se}}$, as shown in Figure \ref{fig:general}. The simplest four ones are given in Figure \ref{fig:basic}. In $\mathcal{T}_2^2$ there are two binary operations: {\it horizontal composition} $\ast$ and {\it vertical composition} $\star$; see Figure \ref{fig:composition}.

\begin{figure} [h]
  \centering
  \includegraphics[width=0.8\textwidth]{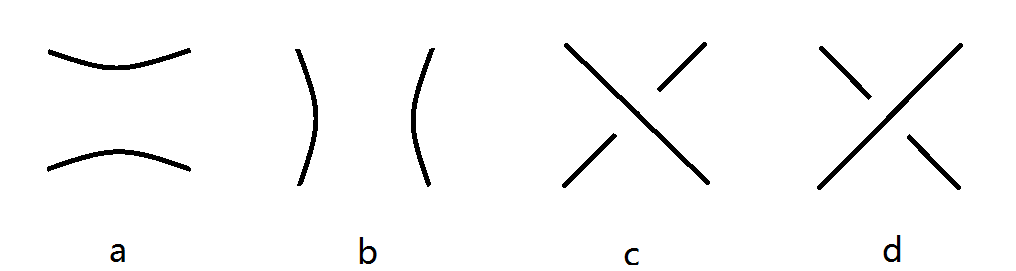}\\
  \caption{The simplest four tangles: (a) $[0]$, (b) $[\infty]$, (c) $[1]$, (d) $[-1]$}\label{fig:basic}
\end{figure}

\begin{figure} [h]
  \centering
  \includegraphics[width=0.6\textwidth]{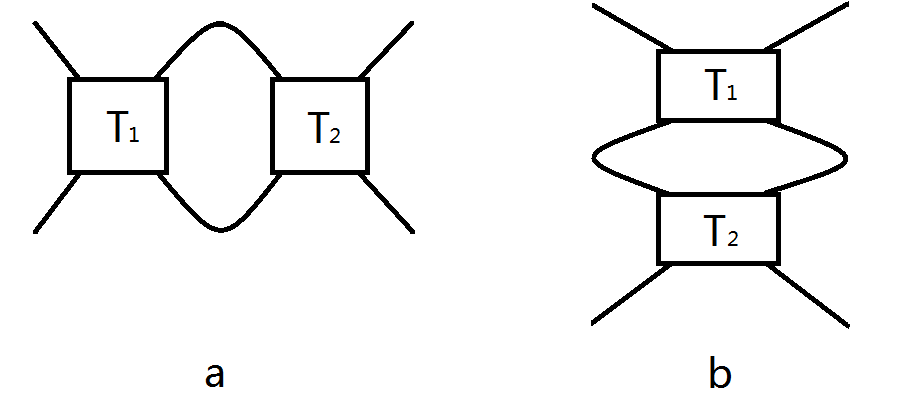}\\
  \caption{(a) $T_1\ast T_2$; (b) $T_1\star T_2$} \label{fig:composition}
\end{figure}

For $k\ne 0$, the horizontal composite of $|k|$ copies of $[1]$ (resp. [-1]) is denoted by $[k]$ if $k>0$ (resp. $k<0$), and the vertical composite of $|k|$ copies of $[1]$ (resp. [-1]) is denoted by $[1/k]$ if $k>0$ (resp. $k<0$).

Given integers $k_{1},\ldots,k_{m}$, we define the \emph{rational tangle}
\begin{align}
[[k_{1}],\ldots,[k_{m}]]=\begin{cases}
[k_{1}]\star [1/k_{2}]\ast\cdots\star [1/k_{m}], &\text{if\ \ } 2\mid m, \\
[k_{1}]\star [1/k_{2}]\ast\cdots\ast [k_{m}], &\text{if\ \ } 2\nmid m,
\end{cases}
\end{align}
and its {\it fraction}
\begin{align}
f([[k_{1}],\ldots,[k_{m}]])=[[k_{1},\ldots,k_{m}]]^{(-1)^{m-1}},  \label{eq:fraction}
\end{align}
where the \emph{continued fraction} $[[k_{1},\ldots,k_{m}]]\in\mathbb{Q}$ is defined inductively as
\begin{align}
[[k_{1}]]=k_{1}, \qquad
[[k_{1},\ldots,k_{m}]]=k_{m}+1/[[k_{1},\ldots,k_{m-1}]].
\end{align}
Denote $[[k_{1}],\ldots,[k_{m}]]$ as $[p/q]$ if the right-hand side of (\ref{eq:fraction}) equals $p/q$.

A \emph{Montesinos tangle} is a tangle of the form
$T=[p_{1}/q_{1}]\star\cdots\star [p_{r}/q_{r}]$. The link obtained by connecting $T^{{\rm nw}}$ and $T^{{\rm ne}}$ with $T^{{\rm sw}}$ and
$T^{{\rm se}}$, respectively, is called a {\it Montesinos link} and denoted by $M(p_1/q_1,\ldots,p_r/q_r)$.

Given a representation $\rho$ of $T\in\mathcal{T}_2^2$, denote
\begin{align}
\rho^{{\rm nw}}=\rho(T^{{\rm nw}}), \qquad &\rho^{{\rm sw}}=\rho(T^{{\rm sw}}), \qquad
\rho^{{\rm ne}}=\rho(T^{{\rm ne}}), \qquad \rho^{{\rm se}}=\rho(T^{{\rm se}}), \\
{\rm tr}_{v}(\rho)&={\rm tr}(\rho^{{\rm ne}}\rho^{{\rm se}}), \qquad
{\rm tr}_{h}(\rho)={\rm tr}(\rho^{{\rm sw}}\rho^{{\rm se}}).
\end{align}

\begin{rmk} \label{rmk:trace-v-h}
\rm Take a 3-ball $\mathcal{B}$ containing $T$ such that $\partial\mathcal{B}$ intersects $T$ precisely at its end points. Then $\pi_1(\mathcal{B}-T)$ has a presentation of Wirtinger type, and $\rho$ can be identified with a representation $\pi_1(\mathcal{B}-T)\to{\rm SL}_0(2,\mathbb{C})$; (this is true for all tangle). Noting that the element of $\pi_1(\mathcal{B}-T)$ presented by $T^{{\rm nw}}T^{{\rm sw}}$ is conjugate to that presented by $T^{{\rm ne}}T^{{\rm se}}$, we also have ${\rm tr}_{v}(\rho)={\rm tr}(\rho^{{\rm nw}}\rho^{{\rm sw}})$.
Similarly ${\rm tr}_{h}(\rho)={\rm tr}(\rho^{{\rm nw}}\rho^{{\rm ne}})$.
\end{rmk}

\section{Representations of rational tangles}

Suppose $[[k_{1},\ldots,k_{m}]]=[p/q]$ with $p,q\ne 0$.
For a representation $\rho$ of $[p/q]$, let
$X,Y,X_{(j)},Y_{(j)}$ denote the elements that $\rho$ assigns to the directed arcs shown in Figure \ref{fig:rational}.
Call $(X,Y)$ the {\it generating pair} of $\rho$, indicating that $\rho$ is determined by $X$ and $Y$.

\begin{figure} [h]
  \centering
  \includegraphics[width=0.45\textwidth]{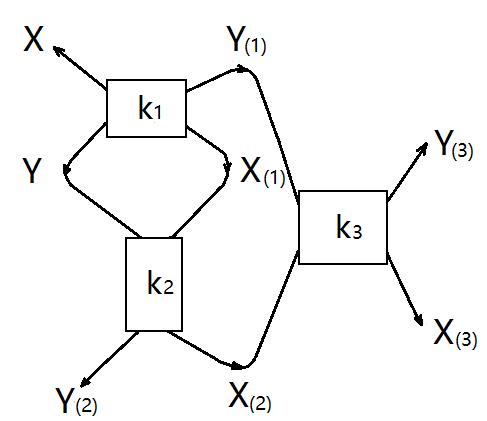} \\
  \caption{A representation of the rational tangle $[[k_{1}],[k_{2}],[k_{3}]]$} \label{fig:rational}
\end{figure}

Suppose ${\rm tr}(XY)=t$. Then $XYX^{-1}=-tX-Y$, hence each $X_{(j)}$ or $Y_{(j)}$ can be expressed as $\alpha(t)X+\beta(t)Y$, where $\alpha(t),\beta(t)$ are polynomials in $t$ that do not depend on $X,Y$.
We take a clever approach to derive formulas for the coefficients $\alpha(t),\beta(t)$.

For $a\in\mathbb{C}^{\times}$, put
\begin{align}
A(a)&=\frac{1}{2}\left(\begin{array}{cc}
(a+a^{-1})i & a-a^{-1} \\
a-a^{-1} & -(a+a^{-1})i \\
\end{array}\right).
\end{align}

\begin{lem}  \label{lem:standard}
{\rm(i)} If $\rho$ is a representation of $[k]$ with $\rho^{{\rm nw}}=A(a_1)$ and $\rho^{{\rm sw}}=A(a_2)$,
then $\rho^{{\rm ne}}=A(-a_1^{k+1}/a_2^k)$ and $\rho^{{\rm se}}=A(-a_1^{k}/a_2^{k-1})$.
\medskip

{\rm(ii)} If $\rho$ is a representation of $[1/k]$ with
$\rho^{{\rm nw}}=A(b_1)$ and $\rho^{{\rm ne}}=A(b_2)$,
then $\rho^{{\rm sw}}=A(-b_1^{k+1}/b_2^k)$ and $\rho^{{\rm se}}=A(-b_1^k/b_2^{k-1})$.
\end{lem}

\begin{proof}
(i) Suppose $\rho$ is a representation of $[1]$ with $\rho^{{\rm nw}}=A(a_1)$ and $\rho^{{\rm sw}}=A(a_2)$, then
$\rho^{{\rm se}}=A(-a_1)$, and computing directly,
$$\rho^{{\rm ne}}=\rho^{{\rm se}}(-\rho^{{\rm sw}})(\rho^{{\rm se}})^{-1}=A(-a_1)A(-a_2)A(-a_1)^{-1}=A(-a_1^2/a_2).$$
Applying this repeatedly, we obtain the result.

(ii) The proof is similar.
\end{proof}

Suppose $\rho$ is a representation of $[p/q]$, with $X=\rho^{{\rm nw}}=A(1)$ and $Y=A(s)$.
By Lemma \ref{lem:standard}, $Y_{(1)}=A(-s^{-k_{1}})$, $X_{(1)}=A(-s^{1-k_1})$, and in general, $Y_{(j)}=A(s_{(j)}), X_{(j)}=A(s'_{(j)})$ with
\begin{align*}
s_{(0)}=s, \qquad  s_{(1)}=-s^{-k_{1}}, \qquad  s'_{(1)}=-s^{1-k_{1}},  \\
\frac{s_{(j)}}{s_{(j-2)}}=\frac{s'_{(j)}}{s'_{(j-1)}}=\left(s_{(j-2)}/s'_{(j-1)}\right)^{k_j}, \qquad j\ge 2.
\end{align*}
Consequently, for $j\ge 2$,
\begin{align*}
s'_{(j)}=\frac{s_{(j)}s_{(j-1)}}{s_{(1)}s_{(0)}}s'_{(1)}=s_{(j)}s_{(j-1)}, \qquad
s_{(j)}=(s_{(j-1)})^{-k_{j}}s_{(j-2)}.
\end{align*}
Define $u_{j}, v_{j}$, $0\le j\le m$, inductively by
\begin{align*}
u_{0}&=0,  & &u_{1}=1,      & &u_{j+1}=k_{j+1}u_{j}+u_{j-1},  \ \ (j\ge 1), \\
v_{0}&=1,  & &v_{1}=k_{1},  & &v_{j+1}=k_{j+1}v_{j}+v_{j-1},  \ \ \ (j\ge 1),
\end{align*}
so that $u_{j}/u_{j-1}=[[k_{2},\ldots,k_{j}]]$, $v_{j}/v_{j-1}=[[k_{1},\ldots,k_{j}]]$.
Then
\begin{align}
Y_{(j)}=A\left((-1)^{u_{j}}s^{(-1)^{j}v_{j}}\right),  \qquad
X_{(j)}=A\left((-1)^{u_{j}+u_{j-1}}s^{(-1)^{j}(v_{j}-v_{j-1})}\right).
\end{align}
Put  $\left\{\begin{array}{ll}
\tilde{p}=u_m,  \ \tilde{q}=u_{m-1}, &\text{if\ } 2\nmid m, \\
\tilde{q}=u_{m}, \ \tilde{p}=u_{m-1}, &\text{if\ } 2\mid m,
\end{array}\right.$
i.e.,
\begin{align}
\tilde{p}/\tilde{q}=[[k_{2},\ldots,k_{m}]]^{(-1)^{m-1}}.  \label{eq:tilde}
\end{align}
One can prove by induction on $m$ that
\begin{align}
\tilde{p}q-p\tilde{q}=1.    \label{eq:p-q}
\end{align}
We have
\begin{align}
\rho^{{\rm ne}}=A((-1)^{\tilde{p}}s^{-p}),  \quad
\rho^{{\rm sw}}=A((-1)^{\tilde{q}}s^{q}),   \quad
\rho^{{\rm se}}=A((-1)^{\tilde{p}+\tilde{q}}s^{q-p}).
\end{align}
Hence
\begin{align}
-{\rm tr}_{v}(\rho)=(-1)^{\tilde{q}}(s^{q}+s^{-q}),  \qquad
-{\rm tr}_{h}(\rho)=(-1)^{\tilde{p}}(s^{p}+s^{-p}).   \label{eq:tr-h}
\end{align}

For an integer $k$, denote
\begin{align}
\{k\}_s={\rm sign}(k)\cdot\sum\limits_{j=1}^{|k|}s^{2j-1-|k|}=
\left\{\begin{array}{ll}
(s^k-s^{-k})/(s-s^{-1}), &s\notin\{\pm 1\}, \\
ks^{k-1}, &s\in\{\pm 1\};
\end{array}\right.
\end{align}
it can be written as a polynomial in $s+s^{-1}$.
Noticing
\begin{align*}
A(s^k)=\{1-k\}_s\cdot X+\{k\}_s\cdot Y,
\end{align*}
we obtain
\begin{align}
\rho^{{\rm ne}}&=(-1)^{\tilde{p}}(\{1+p\}_s\cdot X+\{-p\}_s\cdot Y),   \label{eq:ne}  \\
\rho^{{\rm sw}}&=(-1)^{\tilde{q}}(\{1-q\}_s\cdot X+\{q\}_s\cdot Y),   \label{eq:sw} \\
\rho^{{\rm se}}&=(-1)^{\tilde{p}+\tilde{q}}(\{1+p-q\}_s\cdot X+\{q-p\}_s\cdot Y),  \label{eq:se}
\end{align}
and when $\{p\}_s\ne 0$,
\begin{align}
(\rho^{{\rm sw}},\rho^{{\rm se}})=-(\rho^{{\rm nw}},\rho^{{\rm ne}})\cdot\frac{(-1)^{\tilde{q}-1}}{\{p\}_s}
\left(\begin{array}{cc}
\{p+q\}_s & (-1)^{\tilde{p}}\{q\}_s \\
-(-1)^{\tilde{p}}\{q\}_s & \{p-q\}_s \\
\end{array}\right).    \label{eq:linear}
\end{align}

As pointed out in the second paragraph of this section, these relations are actually valid for arbitrary $X,Y$.

\begin{nota}
\rm For $P,X\in{\rm SL}(2,\mathbb{C})$, denote $PXP^{-1}$ by $P.X$.
\end{nota}

\begin{rmk}   \label{rmk:regular}
\rm For $Z,W\in{\rm SL}_0(2,\mathbb{C})$, call $(Z,W)$ {\it regular} if there exists $P\in{\rm SL}(2,\mathbb{C})$ such that $P.Z=A(1)$ and $P.W=A(s)$ with $s+s^{-1}=-{\rm tr}(ZW)$. It is easy to see that $(Z,W)$ is non-regular if and only if
$-{\rm tr}(ZW)=2a$ where $a\in\{\pm 1\}$ and $W\ne aZ$, and under this condition, there exists $P\in{\rm SL}(2,\mathbb{C})$ such that $P.Z=A(1)$ and $P.W\in\{S_a, S'_a\}$, where
\begin{align}
S_a=\left(\begin{array}{cc} ai & 1 \\ 0 & -ai \\ \end{array}\right),  \qquad
S'_a=\left(\begin{array}{cc} ai & 0 \\ 1 & -ai \\ \end{array}\right).
\end{align}

According to (\ref{eq:ne})--(\ref{eq:se}), the four pairs $(\rho^{{\rm nw}},\rho^{{\rm ne}})$, $(X,Y)$,
$(\rho^{{\rm nw}},\rho^{{\rm sw}})$ and $(\rho^{{\rm sw}},\rho^{{\rm se}})$ are simultaneously regular or not.

\end{rmk}

\section{Representations of Montesinos links}

Given a representation $\rho$, say $\rho$ is {\it reducible} if all the elements in ${\rm Im}(\rho)$ have a common eigenvector; in particular, say $\rho$ is {\it abelian} if ${\rm Im}(\rho)$ is abelian. Call $\rho$ {\it irreducible} if it is not reducible.

Given a Montesinos link $L=M(p_{1}/q_{1},\ldots, p_{r}/q_{r})$, denote
\begin{align}
\mu=\sum\limits_{\ell=1}^{r}\frac{q_{\ell}}{p_{\ell}}.
\end{align}
Suppose $\rho$ is a representation of $L$. Let $\rho_{\ell}$ denote its restriction to $[p_{\ell}/q_{\ell}]$, let $(X_{\ell},Y_{\ell})$ denote the generating pair of $\rho_{\ell}$, and assume
$-{\rm tr}(X_{\ell}Y_{\ell})=s_{\ell}+s_{\ell}^{-1}$. Up to conjugacy we may assume $X_1=A(1)$.

By Remark \ref{rmk:trace-v-h}, the $-{\rm tr}_{h}(\rho_{\ell})$'s have a common value. Take $a\in\mathcal{F}$ so that this value equals $a+a^{-1}$, where
\begin{align}
\mathcal{F}=\{e^{i\varphi}\colon 0\le\varphi\le\pi\}\cup\{s\in\mathbb{C}\colon |s|<1\}.
\end{align}
Then for each $\ell$,
by (\ref{eq:tr-h}),
$$(-1)^{\widetilde{p_{\ell}}}(s_{\ell}^{p_\ell}+s_{\ell}^{-p_\ell})=a+a^{-1},$$
with $\widetilde{p_{\ell}}/\widetilde{q_{\ell}}$ defined as in (\ref{eq:tilde});
switching $s_{\ell}$ with $s_\ell^{-1}$ if necessary, we have
\begin{align}
(-1)^{\widetilde{p_{\ell}}}s_{\ell}^{p_{\ell}}=a.   \label{eq:s-ell}
\end{align}

If $\rho$ is reducible and non-abelian, then $(\rho_{\ell}^{{\rm nw}},\rho_{\ell}^{{\rm ne}})$ is non-regular for at least one of the $\ell$'s, which is, by Remark \ref{rmk:regular}, equivalent to the condition that the
$(\rho_{\ell}^{{\rm nw}},\rho_{\ell}^{{\rm ne}})$'s are all non-regular. Hence, $a\in\{\pm 1\}$ and  $s_{\ell}\in\{\pm 1\}$ with $(-1)^{\widetilde{p_{\ell}}}s_{\ell}^{p_{\ell}}=a$. By (\ref{eq:ne}), $\rho_\ell^{{\rm ne}}$ and $Y_\ell$ are determined by each other; by (\ref{eq:linear}),
\begin{align}
(\rho_\ell^{{\rm sw}},\rho_\ell^{{\rm se}})&=-(\rho_\ell^{{\rm nw}},\rho_\ell^{{\rm ne}})
\cdot(-1)^{\widetilde{q_\ell}-1}s_{\ell}^{q_\ell}B(q_\ell/p_\ell), \label{eq:relation} \\
\text{with} \qquad B(w)&=\left(\begin{array}{ll} 1+w & a w \\ -a w & 1-w \\ \end{array}\right).
\end{align}
Observing  $B(w)B(w')=B(w+w')$, we have $\mu=0$ and
\begin{align}
\prod\limits_{\ell=1}^{r}(-1)^{\widetilde{q_\ell}-1}s_{\ell}^{q_{\ell}}=1.   \label{eq:reducible}
\end{align}
Conversely, when $a,s_1,\ldots,s_r\in\{\pm 1\}$ satisfy $(-1)^{\widetilde{p_{\ell}}}s_{\ell}^{p_{\ell}}=a,\ell=1,\ldots,r$ and (\ref{eq:reducible}) holds, an arbitrary $Y_1\ne aA(1)$ gives rise to
$\rho_\ell^{{\rm nw}}, \rho_\ell^{{\rm ne}},\rho_\ell^{{\rm sw}}, \rho_\ell^{{\rm se}}$ and then $X_\ell,Y_\ell$ for all $\ell$ through $(\rho_{\ell+1}^{{\rm nw}}, \rho_{\ell+1}^{{\rm ne}})=-(\rho_\ell^{{\rm nw}}, \rho_\ell^{{\rm ne}})$, (\ref{eq:relation}) and (\ref{eq:ne}).
The $X_\ell,Y_\ell$'s combine to define a non-abelian reducible representation of $L$.

If $a\in\{\pm 1\}$ and $(\rho_{\ell}^{{\rm nw}},\rho_{\ell}^{{\rm ne}})$ is regular for each $\ell$, then $\rho_\ell^{\rm ne}=-a\rho_\ell^{{\rm nw}}$ and $(\rho_\ell^{{\rm nw}},\rho_\ell^{{\rm sw}})$ is also regular. Conversely, given $X_2,\ldots,X_r$ and $s_1,\ldots,s_r$ such that $(-1)^{\widetilde{p_\ell}}s_\ell^{p_\ell}=a\in\{\pm 1\}$ and (adopting the convention that $X_{r+1}=X_1$)
\begin{align}
(X_\ell,X_{\ell+1}) \quad \text{is\ \ regular\ \ and\ } \quad
{\rm tr}(X_\ell X_{\ell+1})=(-1)^{\widetilde{q_\ell}}(s_{\ell}^{q_\ell}+s_\ell^{-q_\ell})     \label{eq:X}
\end{align}
for each $\ell$,  there is a unique representation $\rho$ of $L$ such that
$$\rho_{\ell}^{{\rm nw}}=X_{\ell}, \qquad \rho_\ell^{\rm ne}=-aX_\ell, \qquad
\rho_\ell^{\rm sw}=-X_{\ell+1},  \qquad \rho_\ell^{\rm se}=aX_{\ell+1};$$
actually, $Y_\ell$ is determined by $X_\ell,X_{\ell+1}$ and $s_\ell$ as in (\ref{eq:sw}):
$$(-1)^{\tilde{q_\ell}}(\{1-q_\ell\}_{s_\ell}\cdot X_\ell+\{q_\ell\}_{s_\ell}\cdot Y_\ell)=-X_{\ell+1}.$$

\begin{rmk}
\rm Note that, $\rho$ is abelian if and only if each $s_\ell\in\{\pm 1\}$; in this case $Y_\ell=s_\ell X_\ell$ and $X_{\ell+1}=(-1)^{\widetilde{q_\ell}+1}s_\ell^{q_\ell}X_\ell$.

Suppose $\rho$ is non-abelian so that $s_\ell\ne\pm 1$ for some $\ell$.
Let
\begin{align}
D=\left(\begin{array}{cc}
0 & 1 \\ -1 & 0 \\ \end{array}\right),  \qquad
E=\left(\begin{array}{cc}
0 & i \\ i & 0 \\ \end{array}\right)=A(1)D.
\end{align}
Note that $A(1)^2=D^2=E^2=-I$, and each $X\in{\rm SL}_0(2,\mathbb{C})$ can be  written as $X=aA(1)+bD+cE$ for a unique triple $(a,b,c)\in\mathbb{C}^3$ with $a^2+b^2+c^2=1$.

Suppose
\begin{align}
X_{\ell}=a_{\ell}A(1)+b_{\ell}D+c_{\ell}E,
\end{align}
with
\begin{align}
a_\ell^2+b_\ell^2+c_\ell^2=1. \label{eq:X-ell}
\end{align}
Then $(a_1,b_1,c_1)=(1,0,0)$, and
\begin{align}
(-1)^{\widetilde{q_\ell}}(s_{\ell}^{q_\ell}+s_\ell^{-q_\ell})=
-{\rm tr}(X_{\ell}X_{\ell+1})=2(a_{\ell}a_{\ell+1}+b_{\ell}b_{\ell+1}+c_{\ell}c_{\ell+1}).  \label{eq:tr-ell}
\end{align}
Since $(p_\ell,q_\ell)=1$, we have that $(-1)^{\tilde{q_\ell}}s_\ell^{q_\ell}\in\{\pm 1\}$ if and only if $s_\ell\in\{\pm 1\}$.
Hence the requirement that $(X_\ell,X_{\ell+1})$ is regular is equivalent to
\begin{align}
(a_{\ell+1},b_{\ell+1},c_{\ell+1})=(-1)^{\tilde{q_\ell}}s_\ell^{q_\ell}\cdot(a_{\ell},b_{\ell},c_{\ell})
\qquad\text{if}\quad s_\ell\in\{\pm 1\}.  \label{eq:regular}
\end{align}
Let $h\in\{1,\ldots,r\}$ be the smallest $k$ such that $s_1,\ldots,s_{k-1}\in\{\pm 1\}$ and $s_{k}\notin\{\pm 1\}$.
Then 
$$X_\ell=(-1)^{\tilde{q}_1+\cdots+\tilde{q}_{\ell-1}}s_1^{q_1}\cdots s_{\ell-1}^{q_{\ell-1}}A(1), \qquad 2\le\ell\le h.$$
Switching $s_h$ with $s_h^{-1}$ if necessary, we may assume $s_h\in\mathcal{F}$.
Conjugating by a diagonal matrix if necessary, we may assume
\begin{align}
X_{h+1}=
(-1)^{\tilde{q}_1+\cdots+\tilde{q}_{h}}s_1^{q_1}\cdots s_{h-1}^{q_{h-1}}A(s_h^{q_h}).
\end{align}
Note that $(a_{h+1},b_{h+1},c_{h+1})$ is determined by this equality.

It is easy to see how to find all of the $(a_\ell,b_\ell,c_\ell)$'s satisfying (\ref{eq:X-ell}), (\ref{eq:tr-ell}) and (\ref{eq:regular}), so as to find all of the tuples $(X_1,\ldots,X_r)$ satisfying (\ref{eq:X}).
\end{rmk}

In the remaining part, suppose $\rho$ is irreducible and $a\notin\{\pm 1\}$; we can assume $Y_1=A(s_1)$.
Write $a=|a|e^{i\theta}$. Then, for each $\ell$,
\begin{align}
s_{\ell}=|a|^{1/p_{\ell}}
\exp\left(\frac{i}{p_{\ell}}(\theta+\widetilde{p_{\ell}}\pi+2n_{\ell}\pi)\right) \qquad  \text{for\ \ some}\quad n_{\ell}\in\mathbb{Z}.
\end{align}
By (\ref{eq:linear}),
\begin{align}
(\rho^{{\rm sw}},\rho^{{\rm se}})=-(\rho^{{\rm nw}},\rho^{{\rm ne}})C((-1)^{\widetilde{q_\ell}-1}s_\ell^{q_\ell}),
\end{align}
with
\begin{align}
C(w)=\frac{1}{a-a^{-1}}\left(\begin{array}{ll} aw-a^{-1}w^{-1} & w-w^{-1} \\ -(w-w^{-1}) & aw^{-1}-a^{-1}w \\ \end{array}\right).
\end{align}
Observing $C(w)C(w')=C(ww')$, we are led to
\begin{align*}
\prod\limits_{\ell=1}^{r}(-1)^{\widetilde{q_{\ell}}-1}s_{\ell}^{q_{\ell}}=1,
\end{align*}
which is, by (\ref{eq:p-q}), equivalent to the pair of equations
\begin{align}
|a|^{\mu}=1
\end{align}
and
\begin{align}
\mu\theta+\pi\sum\limits_{\ell=1}^{r}
\frac{2n_{\ell}q_{\ell}+1}{p_{\ell}}=(2n+r)\pi \qquad \text{for\ some} \quad  n\in\mathbb{Z}.
\end{align}

If $\mu\ne 0$, then $|a|=1$  and
\begin{align}
\theta=\frac{\pi}{\mu}\left(2n+r-\sum\limits_{\ell=1}^{r}\frac{2n_{\ell}q_{\ell}+1}{p_{\ell}}\right)\in(2k\pi,(2k+1)\pi) \qquad
\text{for\ \ some\ \ }k\in\mathbb{Z}.
\end{align}
We may assume
\begin{align}
0\le n_\ell<p_\ell, \quad \ell=1,\ldots,r  \qquad \text{and} \qquad 0\le n<N(\mu),   \label{ineq:n}
\end{align}
where $N(\mu)$ is the numerator of $\mu$.

If $\mu=0$, then $a$ can be arbitrary, and $(n_{1},\ldots,n_{r})$ should satisfy
\begin{align}
\sum\limits_{\ell=1}^{r}\frac{2n_{\ell}q_\ell+1}{p_{\ell}}=2n+r.   \label{eq:final}
\end{align}

\begin{thm} \label{thm:main}
Each conjugacy class of trace-free representations of the Montesinos link $M(p_{1}/q_{1},\ldots p_{r}/q_{r})$ contains a unique representation
$\rho$ such that $X_1=A(1)$, $-{\rm tr}_h(\rho_1)=a+a^{-1}$ with $a\in\mathcal{F}$ and
\begin{enumerate}
  \item[\rm(i)] if $\rho$ is abelian, then it is determined by $a\in\{\pm 1\}$ and a unique tuple $(s_1,\ldots,s_r)\in\{\pm 1\}^{r}$ satisfying {\rm(\ref{eq:s-ell})};
  \item[\rm(ii)] if $\rho$ is reducible but not abelian, then $\mu=0$ and up to the two choices $Y_1\in\{S_a,S'_a\}$,
       $\rho$ is determined by $a\in\{\pm 1\}$ and a unique tuple $(s_1,\ldots,s_r)\in\{\pm 1\}^{r}$ satisfying {\rm(\ref{eq:s-ell})} and {\rm(\ref{eq:reducible})};
  \item[\rm(iii)] if $\rho$ is irreducible with $a\in\{\pm 1\}$, then $\rho$ is determined by $a$ and a unique tuple
       $(h;s_1,\ldots,s_r;a_2,b_2,c_2;\ldots;a_r,b_r,c_r)$ such that $s_1,\ldots,s_{h-1}\in\{\pm 1\}$, $s_h\in\mathcal{F}-\{\pm 1\}$,
       $X_{h+1}=(-1)^{\tilde{q}_1+\cdots+\tilde{q}_{h}}s_1^{q_1}\cdots s_{h-1}^{q_{h-1}}A(s_h)$, and {\rm(\ref{eq:s-ell})}, {\rm(\ref{eq:X-ell})}, {\rm(\ref{eq:tr-ell})}, {\rm(\ref{eq:regular})} hold;
  \item[\rm(iv)] if $\mu=0$ and $\rho$ is irreducible with $a\notin\{\pm 1\}$, then $Y_1=A(s_1)$, and $\rho$ is determined by $a$ and
       a unique tuple $(n,n_1,\ldots,n_r)\in\mathbb{Z}^{r+1}$ satisfying {\rm(\ref{ineq:n})} and {\rm(\ref{eq:final})};
  \item[\rm(v)] if $\mu\ne 0$ and $\rho$ is irreducible with $a\notin\{\pm 1\}$, then $Y_1=A(s_1)$, $|a|=1$ and $\rho$
       is determined by a unique tuple $(n,n_1,\ldots,n_r)\in\mathbb{Z}^{r+1}$ satisfying {\rm(\ref{ineq:n})} and $2n+r-\sum\limits_{\ell=1}^{r}(2n_{\ell}q_{\ell}+1)/p_{\ell}\in(2k\mu,(2k+1)\mu)$ for some $k\in\mathbb{Z}$.
\end{enumerate}
\end{thm}

\begin{rmk}
\rm As pointed out in \cite{Nag13} (see Page 2), the case (ii) never occur if $M(p_1/q_1,\ldots,p_r/q_r)$ is a knot.
\end{rmk}

\begin{rmk}
\rm
Based on this classifying result, without too much difficulty, one may determine the trace-free ${\rm SU}(2)$-representations of a Montesinos link.

Suppose $\varrho$ is a trace-free ${\rm SU}(2)$-representation of $M(p_{1}/q_{1},\ldots p_{r}/q_{r})$. Let $\widetilde{X}_\ell, \widetilde{Y}_\ell$ denote the generating pair of the restriction of $\varrho$ to $[p_\ell/q_\ell]$. Up to conjugacy we may assume $\widetilde{X}_1=A(1)$.

If $\varrho$ is reducible, then it must be abelian, and the result is the same as Theorem \ref{thm:main} (i).

Now suppose $\varrho$ is irreducible. Since $\varrho$ is already a trace-free ${\rm SL}(2,\mathbb{C})$-representation, by Theorem \ref{thm:main} there exists a (diagonal) $P\in{\rm SL}(2,\mathbb{C})$ such that $\rho=P.\varrho$ (the homomorphism sending each $x\in\pi_1(S^3-L)$ to $P.\varrho(x)$) such that $X_1=A(1)$, $-{\rm tr}_h(\rho_1)=a+a^{-1}$ with $a\in\mathcal{F}$, and one of (iii)--(v) holds. Since $\varrho$ takes values in ${\rm SU}(2)$, we have $a=e^{i\varphi}$ with $0\le\varphi\le \pi$.
In case (iii), since both $\widetilde{X}_{h+1}$ and $X_{h+1}=P.\widetilde{X}_{h+1}$ are unitary, we have that actually $P$ is unitary, hence $a_\ell,b_\ell,c_\ell\in\mathbb{R}$, and $\varrho$ is conjugate to $\rho$ as ${\rm SU}(2)$-representations.
Similarly, in case (iv) or (v), $\varrho$ is also conjugate to $\rho$ as ${\rm SU}(2)$-representations.
\end{rmk}


\begin{thebibliography}{}

\bibitem{FKC13}
Y. Fukumoto, P. Kirk, J. Pinz\'on-Caicedo, \\
\textsl{Traceless ${\rm SU}(2)$-representations of 2-stranded tangles}. arxiv: 1305.6042.

\bibitem{HHK13}
M. Hedden, C.M. Herald, P. Kirk,
\textsl{The pillowcase and perturbations of traceless representations of knot groups}.
Geom. Topology 18 (2013), no. 1, 211--287.







\bibitem{KM11}
P.B. Kronheimer, T.S. Mrowka,
\textsl{Knot homology groups from instantons}.
J. Topology 4 (2011), no. 4, 835--918.



\bibitem{Lin92}
X.-S. Lin,
\textsl{A knot invariant via representation spaces}.
J. Diff. Geom. 35 (1992), 337--357.





\bibitem{Mag80}
W. Magnus,
\textsl{Rings of Fricke characters and automorphism groups of free groups}.
Math. Z. 170 (1980), 91--103.


\bibitem{Nag13}
F. Nagasato,
\textsl{On the trace-free characters}.
Kyoto University Research Information Repository 1836 (2013), 110--123.


\bibitem{NY12}
F. Nagasato, Y. Yamaguchi,
\textsl{On the geometry of the slice of trace-free ${\rm SL}_{2}(\mathbb{C})$-characters of a knot group}.
Math. Ann. 354 (2012), 967--1002.


\bibitem{Zen11}
R. Zentner,
\textsl{Representation spaces of pretzel knots}.
Algebr. Geom. Topol. 11 (2011), 2941--2970.



\end{thebibliography}
\end{document}